\documentclass[12pt]{amsart}
\usepackage{latexsym}
\usepackage{amsthm}
\usepackage{amssymb}
\usepackage{amsmath}
\usepackage{color}

\newtheorem{thm}{Theorem}[section]
\newtheorem{lem}[thm]{Lemma}

\newtheorem*{conjs}{Conjectures}
\newtheorem{prop}[thm]{Proposition}

\theoremstyle{remark}
    
\theoremstyle{definition}

\numberwithin{equation}{section}

\def\Z {{\mathbb Z}}

\def\T {{\tilde T}}

\def\H {{\mathcal H}}

\def\l {\left\langle}
\def\r {\right\rangle}

\begin{document}

\title[monomial bases]{Monomial bases for the centres of the group algebra and Iwahori--Hecke algebra of $S_4$}
\author{Andrew Francis}
\address{School of Computing and Mathematics, University of Western Sydney, NSW 1797, Australia}
\email[Andrew~Francis]{a.francis@uws.edu.au}
\author{Lenny Jones}
\address{Department of Mathematics, Shippensburg University, Pennsylvania, USA}
\email[Lenny~Jones]{lkjone@ship.edu}
\date{\today}
\begin{abstract}
G.~E.~Murphy showed in 1983 that the centre of every symmetric group algebra has an integral basis consisting of a specific set of monomial symmetric polynomials in the Jucys--Murphy elements.   While we have shown in earlier work that the centre of the group algebra of $S_3$ has exactly three additional such bases, we show in this paper that the centre of the group algebra of $S_4$ has infinitely many bases consisting of monomial symmetric polynomials in Jucys--Murphy elements, which we characterize completely.  The proof of this result involves establishing closed forms for coefficients of class sums in the monomial symmetric polynomials in Jucys--Murphy elements, and {solving several resulting exponential Diophantine equations with the aid of a computer.}

Our initial motivation {was in} finding {integral} bases for the centre of the Iwahori--Hecke algebra, and we address this question also, {by} finding several integral bases of monomial symmetric polynomials in Jucys--Murphy {elements} for the centre of the Iwahori--Hecke algebra of $S_4$.
\end{abstract}
\maketitle

\section{Introduction}\label{Intro}

The appearance of symmetric polynomials in Jucys--Murphy elements in the context of the centre of the integral group ring of the symmetric group $S_n$ was implicit in work of H.~K.~Farahat and G.~Higman~\cite{FH59}, long before A.~-A.~A.~Jucys~\cite{Jucys1966,Jucys1971,Jucys1974}
and G.~E.~Murphy~\cite{Mur83} had independently defined the elements that bear their names.  Farahat and Higman showed that the centre of $\Z S_n$, which we denote here as $Z(\Z S_n)$, is generated over $\Z$ by sums of permutations of the same cycle type, sums that --- once the definition of Jucys--Murphy elements is made --- are readily seen to be elementary symmetric polynomials in such elements.

Let $S_n$ be the symmetric group on $\{1,\dots,n\}$, generated by the simple transpositions $S=\{s_i:=(i\ i+1)\mid 1\le i\le n-1\}$.  We say an expression for an element $w=s_{i_1}\dots s_{i_k}$ of $S_n$ is \emph{reduced} if $k$ is the minimal number of generators required to write $w$.  In this case, we say that $k$ is the \emph{length of $w$}, and we denote it as $\ell(w)$. We will occasionally abbreviate a reduced expression $s_{i_1}\dots s_{i_k}$ to $s_{i_1i_2\cdots i_k}$.
We say $w$ is \emph{increasing} if $i_1<\dots<i_k$.  If $w$ is increasing, it is a Coxeter element of a Young subgroup $S_\lambda$ for some composition $\lambda\vDash n$.
The \emph{shape} of $w$ is the composition $\lambda-1\vDash \ell(w)$, where $\lambda-1$ means that 1 is subtracted from each component of $\lambda$, omitting all resulting zero components.
There may be many increasing elements of the same shape.  For instance, {$s_{134}$ and $s_{145}$} both have shape $(1,2)$.  Increasing elements are in the same conjugacy class if and only if their shapes have the same components, so we may index conjugacy classes by partitions that have the same components as the shapes of their {increasing elements.} A partition in this indexing is known as a \emph{modified cycle type}~\cite[p.131]{Macdonald95}  {(we thank Weiqiang Wang for pointing us to this reference).  We write $w_\mu$ for an increasing element of shape $\mu$, noting that $w_\mu$ is not uniquely defined.  The conjugacy classes of $S_n$ are indexed by modified cycle type $\mu=(\mu_1,\dots,\mu_r)$ such that $|\mu|+r\le n$ (this ensures that $w_\mu\in S_n$).}  For instance in $S_6$ the conjugacy class $C_{(2,1)}$ has increasing elements {$s_{124}$, $s_{125}$, $s_{134}$, $s_{145}$, $s_{235}$, and $s_{245}$}, each having shape $(2,1)$ or $(1,2)$.

Jucys~\cite{Jucys1966,Jucys1971,Jucys1974}
and Murphy~\cite{Mur83} independently defined the elements $L_1:=0$ and for $2\le i\le n$,
\[L_i:=\sum_{1\le k\le i-1}(k\ i).\]
These elements are known as \emph{Jucys--Murphy elements}. Jucys and Murphy both proved, via different methods, that the set of symmetric polynomials in $\{L_1,\dots,L_n\}$ is precisely $Z(\Z S_n)$.

In his 1974 paper~\cite{Jucys1974}, Jucys followed the approach of Farahat and Higman~\cite{FH59}, showing that the sums of class sums of the same length minimal elements (shown by Farahat and Higman to generate $Z(\Z S_n)$) are the elementary symmetric polynomials in the $L_i$.  Since the elementary symmetric polynomials integrally generate the ring {of} symmetric polynomials, this is sufficient to prove the result.  Murphy, in 1983, gave a set of monomial symmetric polynomials in the $L_i$ {that form} an integral basis for $Z(\Z S_n)$~\cite{Mur83}.  The monomial symmetric polynomials $m_\mu$ in the variables $x_1,\dots, x_n$, for $\mu=(\mu_1,\dots,\mu_r)$, a partition with $r\le n$ parts,  are defined by
\[m_\mu(x_1,\dots,x_n):=\sum_{\sigma\in S_n}x_{\sigma(1)}^{\mu_1}\dots x_{\sigma(r)}^{\mu_r}.\]
Murphy showed that $m_\mu$ contains a minimal length element of shape {$\mu$} with coefficient 1, that does not appear in any other $m_\lambda$ for all $\lambda$ such that elements of shape $\lambda$ are contained in $S_n$.  Since the increasing elements of shape $\lambda$ index the conjugacy classes of $S_n$, it follows that the set of monomial symmetric polynomials in {Jucys--Murphy} elements $\{m_\mu\mid w_\mu\in S_n\}$ forms an integral basis for $Z(\Z S_n)$.

The Iwahori--Hecke algebra $\H_n$ of type $A_{n-1}$ is a deformation of the symmetric group algebra $\Z S_n$, generated by the set $\{\T_i\mid s_i\in S\}$ over $\Z[\xi]$ with relations $\T_i\T_j=\T_j\T_i$ if $|i-j|>1$, $\T_i\T_{i+1}\T_i=\T_{i+1}\T_i\T_{i+1}$, for $1\le i<n-2$, and $\T_i^2=1+\xi\T_i$, for $1\le i<n-1$. For $w\in S_n$, we write $\T_w$ to denote $\T_{i_1}\dots\T_{i_k}$ where $s_{i_1}\dots s_{i_k}$ is a reduced expression for $w$.  Note that $\T_w$ is independent of choice of reduced expression.  At $\xi=0$, the Iwahori--Hecke algebra {$\H_n$} is isomorphic to $\Z S_n$.

Many of the approaches to $Z(\Z S_n)$ described above generalize neatly to the centre of the Iwahori--Hecke algebra $\H_n$, which we denote $Z(\H_n)$.  To begin with, $Z(\H_n)$ has an integral basis of ``class elements'' $\{\Gamma_\lambda\mid w_\lambda\in S_n\}$ that specialize to class sums in $Z(\Z S_n)$~\cite{GR97}. These central elements are characterized by the fact that they specialize to a class sum, and apart from that class sum contain no other minimal length elements from any conjugacy classes~\cite{Fmb}.
The {Jucys--Murphy} elements can be directly generalized, by setting $L_1:=0$ and for $2\le i\le n$,
\[L_i:=\sum_{1\le k\le i-1}\T_{(k\ i)}.\]
The main result of both Jucys and Murphy --- that the centre is the set of symmetric polynomials in {Jucys--Murphy} elements --- generalizes to a result known as the Dipper--James conjecture, shown for the semisimple case by R.~Dipper and G.~James in 1987~\cite{DJ87} and in generality by A.~Francis and J.~Graham in 2006~\cite{FG:DJconj2006}.  The result of Jucys giving the elementary symmetric polynomials in {Jucys--Murphy} elements as a sum of class sums has a direct analogue, and so the analogue of Farahat and Higman's generators for the centre also holds~\cite[Prop. 7.4, Cor. 7.6]{FG:DJconj2006}.  The fly in the ointment is that the basis for $Z(\Z S_n)$ given by Murphy does not generalise to a basis for $Z(\H_n)$, even for $Z(\H_3)$~\cite{FJ:newintbasis}.

While it is possible, as a result of the proof of the Dipper--James Conjecture, to construct an integral basis for $Z(\H_n)$ using linear combinations of monomial symmetric polynomials in {Jucys--Murphy} elements~\cite{FJ:newintbasis}, it is still unclear in general whether there exists an integral basis for $Z(\H_n)$ using monomial symmetric polynomials alone. We refer to any basis for $Z(\Z S_n)$ or $Z(\H_n)$ which consists solely of monomial symmetric polynomials in {Jucys--Murphy} elements as a \emph{monomial basis}. Since, upon specialization at $\xi=0$, any monomial basis for $Z(\H_n)$ gives a monomial basis for $Z(\Z S_n)$, we can identify monomial bases for $Z(\H_n)$ by restricting our attention to sets of monomials in $Z(\H_n)$ which correspond to monomial bases for $Z(\Z S_n)$. Using this strategy in~\cite{FJ:newintbasis}, we showed that there are only four such bases for $Z(\Z S_3)$, only one of which ``lifts to" an integral basis for $Z(\H_3)$.

The coefficient of a class element (resp. class sum), in a central element of {$\H_n$ (resp. $\Z S_n$)}, can be conveniently expressed using the inner product defined by setting
\[\l \T_u,\T_v\r:=\begin{cases}1&\text{if }vu=1\\ 0&\text{otherwise.}\end{cases}\]
This specializes at $\xi=0$ to a corresponding inner product on the group algebra.  For an increasing element $w_\lambda$, there is a unique class element $\Gamma_\lambda$ containing $\T_{w_\lambda}$ with non-zero coefficient, and since the class element $\Gamma_{\lambda}$ contains $\T_{w_\lambda}$ and $\T_{w_\lambda^{-1}}$ with coefficient 1 (both $w_\lambda$ and $w_\lambda^{-1}$ are minimal elements in the same conjugacy class), $\l \T_{w_\lambda},h\r$ gives the coefficient of the class element $\Gamma_\lambda$ in the central element $h$.  Similarly, if $h\in Z(\Z S_n)$ then $\l w_\lambda, h\r$ is the coefficient of the {conjugacy} class sum $\underline{C_\lambda}$ in $h$.  [Note: {as mentioned above,} $\lambda$ is the {modified cycle type corresponding to the conjugacy class, or equivalently the partition giving the shapes of} the increasing elements in the class.]
We have the following parity result for coefficients in monomial symmetric polynomials in {Jucys--Murphy} elements for {$\Z S_n$}:
\begin{lem}[{Lemma 6.1,} \cite{FJ:newintbasis}]\label{lem:chequerboard}
{{In $\Z S_n$}, Let $m_\mu$ be the monomial symmetric polynomial in {Jucys--Murphy} elements corresponding to the partition $\mu$, and let $w_\lambda$ be the increasing element of shape $\lambda$.  }

If {$|\lambda|\neq |\mu| \pmod 2$}, then $\l w_\lambda,m_\mu\r = 0$.
\end{lem}
\medskip

In the present paper, we describe completely all monomial bases for {$Z(\Z S_4)$}.  They fall into two related infinite families and eight exceptional cases (Theorem~\ref{thm:S4bases}).  To do this, we obtain explicit expressions for coefficients of the class sums in any monomial symmetric polynomial in {Jucys--Murphy} elements.  This largely involves finding and solving recursive relations among monomials, and takes up the bulk of this paper (Section~\ref{sec:cl.forms.S4}).  These computations have made use of the computer algebra package GAP~\cite{Sch95} with CHEVIE~\cite{GHLMP}, as well as {Maple${}^\mathrm{TM}$~\cite{Maple10}}.
In Section~\ref{sec:bases}, we use these closed forms to find sets of monomials that form bases for {$Z(\Z S_4)$.} This {procedure requires solving a number} of exponential Diophantine equations, sometimes using congruence arguments that would {have been impractical} without {the use of }a computer.  Finally, in Section~\ref{sec:H4+conjs} we find only three monomial bases for {$Z(\H_4)$}, and conjecture that there are no more.
On the basis of extensive computer calculations, we also report some results for {$Z(\Z S_5)$ and $Z(\H_5)$}, including our failure to find \emph{any} monomial bases for {$Z(\H_5)$}, and our conjecture that there are no monomial bases for $Z(\H_n)$ when $n\ge 5$.

\section{Closed forms for coefficients of class elements in monomial symmetric polynomials}\label{sec:cl.forms.S4}

We begin with some recursion relations that hold for monomial symmetric polynomials in three commuting variables, which are easy to verify.
{If $\lambda=(\lambda_1,\dots,\lambda_r)$ and {appears} in a subscript, we will often drop the parentheses, for instance writing $m_\lambda$ {as} $m_{\lambda_1,\dots,\lambda_r}$.}
\begin{lem}\label{lem:peeling.mi.mii.miii}
  The following relations hold for monomial symmetric functions of exactly three commuting variables:
  {
  \begin{align*}
    m_i&=m_1m_{i-1}-m_{1,1}m_{i-2}+m_{1,1,1}m_{i-3},\quad i>3\\
    m_{i,i}&=m_{1,1}m_{i-1,i-1}-m_1m_{1,1,1}m_{i-2,i-2}+m_{2,2,2}m_{i-3,i-3},\quad i>3\\
    m_{i,i,i}&=m_{1,1,1}m_{i-1,i-1,i-1},\quad i>1.
  \end{align*}}
\end{lem}

While the above Lemma is general, in what follows we {present} the monomial symmetric polynomials evaluated {only} at the Jucys--Murphy elements $L_1,\dots,L_4$, recalling that $L_1=0$.

\begin{lem}\label{lem:recursions.mi.mii.miii}
{For monomial symmetric polynomials $m_\lambda$ evaluated at $L_1,\dots,L_4$, we have the following recursive formulae:}
\begin{enumerate}
\item\label{lem:eq:recursion.mi}
For $i\ge 7$,
  \[m_i=14m_{i-2}-49m_{i-4}+36m_{i-6}.\]

\item\label{lem:eq:recursion.mii}    For $i\ge 6$, we have
  \[m_{i,i}=8m_{i-1,i-1}-5m_{i-2,i-2}-50m_{i-3,i-3}+36m_{i-4,i-4}+72m_{i-5,i-5}.\]

\item\label{lem:eq:recursion.miii}
  For $i\ge 5$,
  \[m_{i,i,i}=40m_{i-2,i-2,i-2}-144m_{i-4,i-4,i-4}.\]
\end{enumerate}
\end{lem}
\begin{proof}
These are all proved by induction, with base steps easily verified using the data in Tables~\ref{tab:coeffs.mi}, \ref{tab:coeffs.mii}, and \ref{tab:coeffs.miii}.

For \eqref{lem:eq:recursion.mi}, assume the claim is true for $i-1$, $i-2$ and $i-3$ (because we are attempting to prove the result for $i\ge 7$, this means the base step involves verifying that the recursion holds for $m_7$, $m_8$ and $m_9$).  Using {Lemma~\ref{lem:peeling.mi.mii.miii}}, we have
\begin{align*}
  m_{i}&=m_1m_{i-1}-m_{1,1}m_{i-2}+m_{1,1,1}m_{i-3}\\
  &=m_1\left(14m_{i-3}-49m_{i-5}+36m_{i-7}\right)- m_{1,1}\left(14m_{i-4}-49m_{i-6}+36m_{i-8}\right)\\
  &\qquad +m_{1,1,1}\left(14m_{i-5}-49m_{i-7}+36m_{i-9}\right)\\
  &=14(m_1m_{i-3}-m_{1,1}m_{i-4}+m_{1,1,1}m_{i-5})-49(m_1m_{i-5}-m_{1,1}m_{i-6}+m_{1,1,1}m_{i-7})\\
  &\qquad +36(m_1m_{i-7}-m_{1,1}m_{i-8}+m_{1,1,1}m_{i-9})\\
  &=14m_{i-2}-49m_{i-4}+36m_{i-6}.
\end{align*}
The inductions for \eqref{lem:eq:recursion.mii} and \eqref{lem:eq:recursion.miii} are similar.
\end{proof}

Using the recurrence relations in Lemma \ref{lem:recursions.mi.mii.miii}, we derive the following closed forms for coefficients of class sums in the monomials.

\begin{lem}\label{lem:cl.forms.mi.mii.miii}
{For monomial symmetric polynomials $m_\lambda$ evaluated at $L_1,\dots,L_4$, we have the following closed forms for coefficients of class sums (recalling that $\l w_\lambda,m_\mu\r$ is the coefficient of $\underline{C_\lambda}$ in $m_\mu$).}
For $i\ge 1$,
\begin{itemize}
\item {$\l -,m_i(L_1,\dots,L_4)\r$}
  \begin{align*}
    \l 1,m_{i}\r &=\frac{1}{24}\left(1+(-1)^i\right)\left(3^i+10\cdot 2^i+23\right)\\
    \l s_1,m_{i}\r&=\frac{1}{24}\left(1-(-1)^i\right)\left(3^{i}+4\cdot 2^{i}+1\right)\\
    \l s_{12},m_{i}\r&=\frac{1}{24}\left(1+(-1)^i\right)\left(3^{i}+2^{i}-1\right)\\
    \l s_{13},m_{i}\r&=\frac{1}{24}\left(1+(-1)^i\right)\left(3^{i}-2\cdot 2^{i}-1\right)\\
    \l s_{123},m_{i}\r&=\frac{1}{24}\left(1-(-1)^i\right)\left(3^{i}-2\cdot 2^{i}+1\right)\\
  \end{align*}
\item {$\l -,m_{i,i}(L_1,\dots,L_4)\r$}
  \begin{align*}
    \l 1,m_{i,i}\r&=\frac{1}{12}\left(6^i+3^i+10\cdot 2^i+9(-2)^i+11(-1)^i\right)\\
    \l s_1,m_{i,i}\r&=0\\
    \l s_{12},m_{i,i}\r&=\frac{1}{12}\left(6^i+3^i+2^i-(-1)^i\right)\\
    \l s_{13},m_{i,i}\r&=\frac{1}{12}\left(6^i+3^i-2\cdot 2^i-(-1)^i-3(-2)^i\right)\\
    \l s_{123},m_{i,i}\r&=0\\
  \end{align*}
\item {$\l -,m_{i,i,i}(L_1,\dots,L_4)\r$}
  \begin{align*}
    \l 1,m_{i,i,i}\r &=\frac{1}{4}\left(1+(-1)^i\right)\left(6^{i-1}+3\cdot 2^{i-1}\right)\\
    \l s_1,m_{i,i,i}\r      &=\frac{1}{4}\left(1-(-1)^i\right)\left(6^{i-1}-2^{i-1}\right)\\
    \l s_{12},m_{i,i,i}\r   &=\frac{1}{4}\left(1+(-1)^i\right)6^{i-1}\\
    \l s_{13},m_{i,i,i}\r &=\frac{1}{4}\left(1+(-1)^i\right)\left(6^{i-1}-2^{i-1}\right)\\
    \l s_{123},m_{i,i,i}\r  &=\frac{1}{4}\left(1-(-1)^i\right)\left(6^{i-1}+2^{i-1}\right)
  \end{align*}
 \end{itemize}
\end{lem}

\begin{proof}
  We give the details of the proof of the closed form for $\l 1,m_i\r$. The proofs of the other cases are similar. The validity of the closed form for $i=1,\ldots ,6$ is easily verified using Table~\ref{tab:coeffs.mi}. The recurrence from Lemma \ref{lem:recursions.mi.mii.miii} (\ref{lem:eq:recursion.mi}) has characteristic polynomial $x^6-14x^4+49x^2-36$ which has zeros $\pm1 $, $\pm 2$, and $\pm 3$. Therefore, for $i\ge 7$,
\begin{equation}\label{eq:eigen}
\l 1, m_i\r=c_1+c_2(-1)^i+c_32^i+c_4(-2)^i+c_53^i+c_6(-3)^i,\end{equation}
for some constants $c_1,\ldots, c_6$. Imposing the initial conditions from Table \ref{tab:coeffs.mi} for $i=1,\ldots ,6$, we get that
\[c_1=c_2=\frac{23}{24},\qquad c_3=c_4=\frac{5}{12}, \qquad c_5=c_6=\frac{1}{24}.\]
Substituting these values into (\ref{eq:eigen}) and factoring gives the closed form.
\end{proof}

\begin{table}[p]\caption{Coefficients for $m_i$ in $\Z S_4$, $1\le i\le 9$, obtained using GAP \cite{Sch95}.}
\label{tab:coeffs.mi}
\begin{tabular}{c|cccccccccc}
  &$m_0$&$m_1$&$m_2$&$m_3$&$m_4$&$m_5$&$m_6$&$m_7$&$m_8$&$m_9$\\
\hline
  1         &1&0&6&0&22&0&116&0&762&0\\
  $s_1$     &0&1&0&5&0&31&0&225&0&1811\\
  $s_{12}$  &0&0&1&0&8&0&66&0&568&0\\
  $s_{13}$  &0&0&0&0&4&0&50&0&504&0 \\
  $s_{123}$ &0&0&0&1&0&15&0&161&0&1555
\end{tabular}
\end{table}

\begin{table}[h]\caption{Coefficients for $m_{i,i}$ in $\Z S_4$, $1\le i\le 9$, obtained using GAP \cite{Sch95}.}\label{tab:coeffs.mii}
\begin{tabular}{c|ccccccccc}
  &$m_{1,1}$&$m_{2,2}$&$m_{3,3}$&$m_{4,4}$&$m_{5,5}$&$m_{6,6}$&$m_{7,7}$&$m_{8,8}$&$m_{9,9}$\\
\hline
  1         &0&11&20&141&670&4051&23520&140921&841490\\
  $s_1$     &0&0 &0 &0  &0  &0   &0    &0     &0      \\
  $s_{12}$  &1&4 &21&116&671&3954&23521&140536&841491\\
  $s_{13}$  &1&2 &21&108&671&3922&23521&140408&841491\\
  $s_{123}$ &0&0 &0 &0  &0  &0   &0    &0     &0
\end{tabular}
\end{table}

\footnotesize{
\begin{table}[h]\caption{Coefficients for $m_{i,i,i}$ in $\Z S_4$, $1\le i\le 9$, obtained using GAP \cite{Sch95}.}\label{tab:coeffs.miii}
\begin{tabular}{c|ccccccccc}
  &$m_{1,1,1}$&$m_{2,2,2}$&$m_{3,3,3}$&$m_{4,4,4}$&$m_{5,5,5}$&$m_{6,6,6}$&$m_{7,7,7}$&$m_{8,8,8}$&$m_{9,9,9}$\\
\hline
  1         &0&6&0 &120&  0&3936&    0&140160&     0 \\
  $s_1$     &0&0&16&0  &640&   0&23296&     0&839680 \\
  $s_{12}$  &0&3&0 &108&  0&3888&    0&139968&     0 \\
  $s_{13}$  &0&2&0 &104&  0&3872&    0&139904&     0 \\
  $s_{123}$ &1&0&20&0  &656&   0&23360&     0&839936
\end{tabular}
\end{table} }
\normalsize

To address general monomials, we need to consider coefficients in products of monomials.

\begin{lem}\label{lem:mon.prod.redn}
The coefficients of the class sums {from $\Z S_4$} in products of monomial symmetric polynomials in {the Jucys--Murphy} elements
{$L_1,\dots, L_4$} are given by the following:

\begin{align*}
\l 1,m_\mu m_\lambda\r
&=\l 1,m_\mu\r\l1,m_\lambda\r+6\l s_1,m_\mu\r\l s_1,m_\lambda\r+ 3\l s_{13},m_\mu\r\l s_{13},m_\lambda\r\\
&\qquad +8\l s_{12},m_\mu\r\l s_{12},m_\lambda\r+6\l s_{123},m_\mu\r\l s_{123},m_\lambda\r,\\
\l s_1,m_\mu m_\lambda\r
&=\l 1,m_\mu\r\l s_1,m_\lambda\r+\l s_1,m_\mu\r\left(\l 1,m_\lambda\r+\l s_{13},m_\lambda\r+4\l s_{12},m_\lambda\r\right)\\
&\qquad+\l s_{12},m_\mu\r\left(4\l s_1,m_\lambda\r+4\l s_{123},m_\lambda\r\right)\\
&\qquad+\l s_{13},m_\mu\r\left(\l s_1,m_\lambda\r+2\l s_{123},m_\lambda\r\right)\\
&\qquad+\l s_{123},m_\mu\r\left(4\l s_{12},m_\lambda\r+2\l s_{13},m_\lambda\r\right),\\
\l s_{12},m_\mu m_\lambda\r
&=\l 1,m_\mu\r\l s_{12},m_\lambda\r+\l s_1,m_\mu\r\left(3\l s_1,m_\lambda\r+3\l s_{123},m_\lambda\r\right)\\
&\qquad+\l s_{12},m_\mu\r\left(\l 1,m_\lambda\r+3\l s_{13},m_\lambda\r+4\l s_{12},m_\lambda\r\right)\\
&\qquad+3\l s_{13},m_\mu\r\l s_{12},m_\lambda\r+\l s_{123},m_\mu\r\left(3\l s_{1},m_\lambda\r+3\l s_{123},m_\lambda\r\right),\\
\l s_{13},m_\mu,\lambda\r
&=\l 1,m_\mu\r\l s_{13},m_\lambda\r+\l s_1,m_\mu\r\left(2\l s_1,m_\lambda\r+4\l s_{123},m_\lambda\r\right)\\
&\qquad+8\l s_{12},m_\mu\r\l s_{12},m_\lambda\r+ \l s_{13},m_\mu\r\left(\l 1,m_\lambda\r+2\l s_{13},m_\lambda\r\right)\\
&\qquad+\l s_{123},m_\mu\r\left(4\l s_{1},m_\lambda\r+2\l s_{123},m_\lambda\r\right),\\
\l s_{123},m_\mu,\lambda\r
&=\l 1,m_\mu\r\l s_{123},m_\lambda\r+\l s_1,m_\mu\r\left(4\l s_{12},m_\lambda\r+2\l s_{13},m_\lambda\r\right)\\
&\qquad+\l s_{12},m_\mu\r\left(4\l s_{1},m_\lambda\r+4\l s_{123},m_\lambda\r\right)\\
&\qquad+\l s_{13},m_\mu\r\left(2\l s_1,m_\lambda\r+\l s_{123},m_\lambda\r\right)\\
&\qquad+\l s_{123},m_\mu\r\left(\l {1},m_\lambda\r+4\l s_{12},m_\lambda\r+\l s_{13},m_\lambda\r\right),\\
\end{align*}
\end{lem}
\begin{proof}
The expansion for the coefficient of the identity follows by counting sizes of conjugacy classes.  In other cases we must count elements in conjugacy classes that multiply to give the element in question.
For instance, to find the coefficient of $s_1$ in a product of monomials, we must find all pairs of elements in $S_4$ whose product is $s_1$.  A careful listing of such pairs in each case gives rise to the expansions in the statement.
\end{proof}

\begin{prop}
  The coefficients of class sums in monomials of form {$m_{i,i,j}(L_1,\dots,L_4)$} for $i\neq j$ {and $i,j\ge 1$} are given by the following closed forms:
  \begin{align*}
  \l 1,m_{i,i,j}\r      &=\frac{1}{24}(1+(-1)^j)\left(6^i+2^i3^j+2^j3^i+9(-1)^i(2^i+2^j)+9\cdot 2^i\right),\\
  \l s_1,m_{i,i,j}\r    &=\frac{1}{24}(1-(-1)^j)\left(6^i+2^i3^j+2^j3^i+3(-1)^i(2^i+2^j)-3\cdot 2^i\right),\\
  \l s_{12},m_{i,i,j}\r      &=\frac{1}{24}(1+(-1)^j)\left(6^i+2^i3^j+2^j3^i\right),\\
  \l s_{13},m_{i,i,j}\r    &=\frac{1}{24}(1+(-1)^j)\left(6^i+2^i3^j+2^j3^i-3(-1)^i(2^i+2^j)-3\cdot 2^i\right),\\
  \l s_{123},m_{i,i,j}\r    &=\frac{1}{24}(1-(-1)^j)\left(6^i+2^i3^j+2^j3^i-3(-1)^i(2^i+2^j)+3\cdot 2^i\right).
  \end{align*}
\end{prop}
\begin{proof}
These follow from the relation $m_{i,i,j}=m_{j,j,j}m_{i-j,i-j}$ if $i>j$ or $m_{i,i,j}=m_{i,i,i}m_{j-i}$ if $i<j$, together with the product expansions in Lemma~\ref{lem:mon.prod.redn} and the closed forms in Lemma~\ref{lem:cl.forms.mi.mii.miii}. While the reduction is different depending on whether $i$ is {greater} than $j$ or not, the closed forms in terms of $i$ and $j$ have the same expression.
\end{proof}

\begin{prop}
  The coefficients of class sums in monomials of form {$m_{i+j,i}(L_1,\dots,L_4)$ with $i,j\ge 1$} are given by the following closed forms:
  \begin{align*}
    \l 1,m_{i+j,i}\r&=\frac{1}{24}(1+(-1)^j)\left(6^i3^j+6^i2^j+3^{i+j}+3^i\right.\\
                    &\qquad \left.+\left(10+9(-1)^i\right)\left(2^{i+j}+2^i\right)+22(-1)^i\right),\\
    \l s_1,m_{i+j,i}\r&=\frac{1}{24}\left(1-(-1)^j\right)\left(3^j6^i+2^j6^i+3^{i+j}+3^i\right.\\
                            &\qquad\left.+\left(4+3(-1)^i\right)2^{i+j}+\left(4-3(-1)^i\right)2^i\right),\\
    \l s_{12},m_{i+j,i}\r&=\frac{1}{24}\left(1+(-1)^j\right)\left(3^j6^i+2^j6^i+3^{i+j}+3^i+2^{i+j}+2^i-2(-1)^i\right),\\
    \l s_{13},m_{i+j,i}\r &=\frac{1}{24}\left(1+(-1)^j\right)\left(3^j6^i+2^j6^i+3^{i+j}+3^i\right.\\
                        &\qquad\left.-(2+3(-1)^i)(2^i+2^{i+j})-2(-1)^i\right),\\
    \l s_{123},m_{i+j,i}\r&=\frac{1}{24}\left(1-(-1)^j\right)\left(3^j6^i+2^j6^i+3^i+3^{i+j}\right. \\
                        &\qquad\left.+3(-1)^i(2^i-2^{i+j})-2(2^i+2^{i+j})\right).\\
  \end{align*}
\end{prop}

\begin{proof}
  These inner products all follow from the closed forms given in Lemma~\ref{lem:cl.forms.mi.mii.miii}, together with the reduction relations give in Lemma~\ref{lem:mon.prod.redn} and the following relations among monomial symmetric polynomials in exactly three variables:
  \[m_{i+j,i}=m_{i,i}m_j-m_{i,i,j}=\begin{cases}m_{i,i}m_j-m_{j,j,j}m_{i-j,i-j}&i>j\\ m_{i,i}m_j-m_{i,i,i}m_{j-i}&i<j\end{cases}\]
  and
  \[m_{2i,i}=  m_{i,i}m_i-3m_{i,i,i}.\]

For instance, for $i>j$ we have
\begin{align*}
  \l 1,m_{i+j,i}\r
    &=\l 1,m_{i,i}m_j\r-\l 1,m_{j,j,j}m_{i-j,i-j}\r\\
    &=\l 1,m_{i,i}\r\l1,m_j\r+3\l s_{13},m_{i,i}\r\l s_{13},m_j\r+8\l s_{12},m_{i,i}\r\l s_{12},m_j\r-\\
    &\qquad [\l 1,m_{j,j,j}\r\l1,m_{i-j,i-j}\r+3\l s_{13},m_{j,j,j}\r\l s_{13},m_{i-j,i-j}\r+\\
    &\qquad 8\l s_{12},m_{j,j,j}\r\l s_{12},m_{i-j,i-j}\r]\\
    &=\frac{1}{24}(1+(-1)^j)\left(6^i3^j+6^i2^j+3^{i+j}+3^i+\right.\\
    &\qquad \left.\left(10+9(-1)^i\right)\left(2^{i+j}+2^i\right)+22(-1)^i\right).
\end{align*}
Note that while the factorization of $m_{i,i,j}$ into two parts depends on whether $i>j$, $i<j$ or $i=j$, the resulting closed forms are equal and hence we obtain a single expression for all $i,j$.
\end{proof}

\begin{prop}
  The coefficients of class sums in monomials of form {$m_{k+i+j,k+i,k}(L_1,\dots,L_4)$ for $i,j,k\ge 1$} are given by
  \begin{align*}
    \l 1,m_{k+i+j,k+i,k}\r&=\frac{1}{24}\left(1+(-1)^{j+k}\right)\left[6^k\left(3^j6^i+2^j6^i+3^{i+j}+3^i+2^{i+j}+2^i\right)\right.\\
        &\qquad \left.+9\left((-1)^k+(-1)^{i+k}\right)2^{i+j+k}+9\left((-1)^i+(-1)^j\right)2^{i+k}\right.\\
        &\qquad \left.+9\left((-1)^j+(-1)^{j+k}\right)2^k\right],\\
    \l s_1,m_{k+i+j,k+i,k}\r&=\frac{1}{24}\left(1-(-1)^{j+k}\right)\left[6^k\left(3^j6^i+2^j6^i+3^{i+j}+3^i+2^{i+j}+2^i\right)\right.\\
        &\qquad \left.+3\left((-1)^k+(-1)^{i+k}\right)2^{i+j+k}+3\left((-1)^k-(-1)^i\right)2^{i+k}\right.\\
        &\qquad \left.+3\left((-1)^{i+k}-(-1)^i\right)2^k\right],\\
    \l s_{12},m_{k+i+j,k+i,k}\r&=\frac{1}{24}\left(1+(-1)^{j+k}\right)6^k\left(3^j6^i+2^j6^i+3^{i+j}+3^i+2^{i+j}+2^i\right),\\
    \l s_{13},m_{k+i+j,k+i,k}\r&=\frac{1}{24}\left[\left(1+(-1)^{j+k}\right)6^k\left(3^j6^i+2^j6^i+3^{i+j}+3^i+2^{i+j}+2^i\right)\right.\\
        &\qquad -3\left((-1)^j+(-1)^k+(-1)^{i+j}+(-1)^{i+k}\right)2^{i+j+k}\\
        &\qquad -3\left((-1)^i+(-1)^{j}+(-1)^{k}+(-1)^{i+j+k}\right)2^{i+k}\\
        &\qquad \left.-3(-1)^i\left(1+(-1)^{j}+(-1)^{k}+(-1)^{j+k}\right)2^k\right],\\
    \l s_{123},m_{k+i+j,k+i,k}\r&=\frac{1}{24}\left[\left(1-(-1)^{j+k}\right)6^k\left(3^j6^i+2^j6^i+3^{i+j}+3^i+2^{i+j}+2^i\right)\right.\\
        &\qquad +3\left((-1)^j-(-1)^k+(-1)^{i+j}-(-1)^{i+k}\right)2^{i+j+k}\\
        &\qquad +3\left((-1)^i+(-1)^{j}-(-1)^{k}-(-1)^{i+j+k}\right)2^{i+k}\\
        &\qquad \left.+3(-1)^i\left(1+(-1)^{j}-(-1)^{k}-(-1)^{j+k}\right)2^k\right].\\
  \end{align*}
\end{prop}

\begin{proof}
  {For monomial symmetric polynomials of three (non-zero) variables, we have} $m_{i+j+k,i+k,k}=m_{k,k,k}m_{i+j,i}$. {Combining this with}  the formulas for coefficients in products of monomials given in Lemma~\ref{lem:mon.prod.redn} {gives the result}.
\end{proof}

\section{Bases consisting of monomials}\label{sec:bases}

An integral basis for $Z(\Z S_4)$ has a transition matrix to the minimal {(class sum)} basis that is invertible over the integers.  In other words, the determinant of the transition matrix is $\pm 1$. Because of Lemma~\ref{lem:chequerboard}, we can reorder the bases so that the transition matrix is block diagonal, with an ``even'' block and an ``odd'' block.  The even ($3\times 3$) block consists of monomials whose partition is of an even integer, and {therefore} whose class elements have shortest element either 1, $s_{12}$ or $s_{13}$. We refer to these monomials as \emph{even monomials}. The odd ($2\times 2$) block
consists of monomials whose partition is of an odd integer, and whose class elements have shortest element either $s_1$ or $s_{123}$. We refer to these monomials as \emph{odd monomials}. That is, the transition matrix of these reordered bases is of form
\[\bordermatrix{
            &1  &s_{12} &s_{13}&s_{1}&s_{123}\cr
   1 &\ast       &\ast&\ast &   &\cr
  s_{12}    &\ast       &\ast&\ast &   &\cr
  s_{13}    &\ast       &\ast&\ast &   &\cr
  s_{1}     &           &    &     &\ast&\ast\cr
  s_{123}   &           &   &      &\ast&\ast
}.
\]

Some monomials have even coefficients on all class elements, and therefore they cannot be part of a basis (the corresponding column in the transition matrix would have a factor of 2). Using the closed forms for coefficients on class elements obtained above, we can determine which monomials have this property, and rule them out immediately as possible basis elements. We consider first the even monomials:
\begin{itemize}
\item $m_\emptyset$,
\item $m_{i,i}$ for all $i$
\item $m_{i}$ and $m_{i,i,i}$ for $i$ even
\item $m_{i+j,i}$ and $m_{i,i,j}$ for $j$ even and
\item $m_{k+i+j,k+i,k}$ for $j+k$ even.
\end{itemize}

\begin{lem}\label{lem:evenmoncan} The even monomials that have at least one odd coefficient on a class element
are as follows:
\begin{enumerate}
\item $m_\emptyset$
\item $m_{i,i}$ for all $i$
\item $m_2$ from $\{m_i\mid i\text{ even}\}$ \label{mieven}
\item $m_{2,2,2}$ from $\{m_{i,i,i}\mid i\text{ even}\}$
\item $m_{i,i,2}$ from $\{m_{i,i,j}\mid j\text{ even, } j\ne i\}$ \label{miijeven}
\item none from $\{m_{i+j,i}\mid j\text{ even}\}$
\item none from $\{m_{k+i+j,k+i,k}\mid j+k\text{ even}\}$.
\end{enumerate}
\end{lem}

\begin{proof}

 Since the proofs are similar, we present only the proofs of parts (\ref{mieven}) and (\ref{miijeven}).

To prove part (\ref{mieven}), we first observe that $\l 1,m_i\r$ is even
if and only if $3^i+10\cdot 2^i+23 \equiv 0 \pmod{8}$ (see Lemma~\ref{lem:cl.forms.mi.mii.miii}).  Since $i$ is even, $3^i\equiv 1\pmod 8$, and $10\cdot 2^i\equiv 0\pmod 8$, this coefficient is always even.

Similarly, the coefficient on $s_{12}$ is even if and only if $3^i+2^i-1\equiv 0\pmod 8$.  This congruence holds if and only if $i\ge 4$, meaning that $m_2$ has odd coefficient on $s_{12}$ but no others in this family do.

 Finally, the coefficient on $s_{13}$ is even if and only if $3^i-2\cdot 2^i-1\equiv 0\pmod 8$.  With $i$ even, we have $3^i\equiv 1\pmod 8$ and $2^{i+1}\equiv 0\pmod 8$. So this coefficient is even for all even $i$.

  Thus, with the exception of $m_2$, all monomials in this family have even coefficients on all class elements, making $m_2$ the only candidate for inclusion in a basis.

For part (\ref{miijeven}), it is easy to see that for $j\ge 4$, the expressions \[6^i+2^i3^j+2^j3^i+9(-1)^i(2^i+2^j)+9\cdot 2^i,\] \[6^i+2^i3^j+2^j3^i \quad \text{ and } \quad  6^i+2^i3^j+2^j3^i-3(-1)^i(2^i+2^j)-3\cdot 2^i\] are all divisible by 8, making the coefficients on class sums all even. For example, \[6^i+2^i3^j+2^j3^i\equiv (-2)^i+2^i(1)+0\equiv 2^i(1+(-1)^i)\equiv 0\pmod 8,\] which accounts for the coefficient of $s_{12}$.

  This leaves the case $j=2$, which means that $m_{i,i,2}$ is our only candidate for a basis element from this family. In this case, we have that $\l s_{12}, m_{i,i,2}\r$ is odd since $6^i+2^i3^2j+2^23^i\equiv 4\pmod 8,$ while all other coefficients are even.

\end{proof}

We have the following table of candidate even monomials given in Lemma~ \ref{lem:evenmoncan}:

\medskip
\begin{tabular}{c|cccccc}
&$\emptyset$&$(2)$&$(2,2,2)$&$(2i,2i)$&$(2i+1,2i+1)$&$(i,i,2)$\\
\hline
1&1&6&6&odd&even&even\\
$s_{12}$&0&1&3&even&odd&odd\\
$s_{13}$&0&0&2&even&odd&even
\end{tabular}
\medskip

The parities shown in the table are direct consequences of the closed forms obtained earlier.

Since these columns must form columns of a matrix whose determinant is $\pm 1$, we must have a monomial of form $m_{2i+1,2i+1}$ in a basis (see the coefficients of $s_{13}$).  Looking at the coefficients of the identity, we must also have either $m_\emptyset$ or one of form $m_{2i,2i}$.  Each of these three has the same parity on both $s_{12}$ and $s_{13}$, so we must include a monomial that has different parity on these two: either $m_{2}, m_{2,2,2}$ or $m_{i,i,2}$.  Thus, we have six possible (families of) sets of monomials given below that could form a basis. For each of these cases, we examine the determinant of the $3\times 3$ matrix to ascertain which have a determinant of $\pm 1$.

\begin{itemize}
  \item $\{m_{2i+1,2i+1},m_\emptyset,m_{2}\}$\\
  The determinant of this transition matrix is given by \[\frac{1}{12}\left(6^{2i+1}+3^{2i+1}-(2+3(-1)^{2i+1})2^{2i+1}-(-1)^{2i+1}\right).\]  This determinant is exactly 1 when $i=0$ and is increasing with $i$. Thus, $i=0$ gives the only spanning subset here, namely $\{m_\emptyset,m_{1,1},m_2\}$.
  \item $\{m_{2i+1,2i+1},m_\emptyset,m_{2,2,2}\}$\\
  The determinant here is \[\frac{1}{12}\left(6^{2i+1}+3^{2i+1}-(8+9(-1)^{2i+1})2^{2i+1}-(-1)^{2i+1}\right),\] which is 1 when $i=0$ and increasing with $i$.  Thus, $\{m_\emptyset,m_{1,1},m_{2,2,2}\}$ is the only spanning set from this family.
  \item $\{m_{2i+1,2i+1},m_\emptyset,m_{j,j,2}\}$\\
  The determinant simplifies to
\[-\frac{1}{48}\left(6^{2i+1}+3^{2i+1}+2^{2i+1}+1\right)\left[\left(1+(-1)^j\right)2^j+4(-1)^j\right]\]
When $j$ is even, we have that $\left(1+(-1)^j\right)2^j+4(-1)^j>4$. Then, since $6^{2i+1}+3^{2i+1}+2^{2i+1}+1\ge 12$, there are no values of $i$ and $j$ that give determinant $\pm 1$. When $j$ is odd, the determinant is 1 if and only if $i=0$.  Thus the set $\{m_\emptyset,m_{1,1},m_{j,j,2}\mid j\text{ odd}\}$ gives all sets of this form that span the class elements here.
  \item $\{m_{2i+1,2i+1},m_{2j,2j},m_2\}$\\
  The determinant simplifies to \[\frac{1}{12}\left(\left(2^{2j}-1\right)\left(6^{2i+1}+3^{2i+1}+2^{2i+1}\right) +6\cdot 2^{2j} -6^{2j}-3^{2j}\right).\] We claim that this determinant is never equal to $-1$, and is equal to $1$ only when $i=0$ and $j=1$. To see this, we rewrite the determinant, replacing $2j$ with {$w+3$}, and set it equal to $\pm 1$. Multiplying by 12 produces the following two exponential
Diophantine equations:
\begin{equation}\label{Dio1}
\left(2^{w+3}-1\right)\left(6^i+3^i+2^i\right) +6\cdot 2^{w+3} -6^{w+3}-3^{w+3}=\pm 12.
\end{equation}
In the case when the right-hand side of (\ref{Dio1}) is $-12$, the equation has no solutions modulo 28. When the right-hand side of (\ref{Dio1}) is $12$, we see that there are no solutions by reducing the equation modulo 1971. All calculations were done by computer. Thus, the equations in (\ref{Dio1}) have no solutions for any integers $i\ge 1$ and $w\ge 0$. This implies that the original determinant can only be $\pm 1$ when $j=1$. Substituting $j=1$ into the original determinant shows that $i$ must be {0}.
Hence, the only spanning set here is $\{m_{1,1},m_{2,2},m_{2}\}$.
    \item $\{m_{2i+1,2i+1},m_{2j,2j},m_{2,2,2}\}$\\
    The determinant simplifies to \[\frac{1}{12}\left(17\cdot 2^{2j}-6^{2i+1}-6^{2j}-3^{2i+1}-3^{2j}-2^{2i+1}\right),\] which is less than $-1$ when $j>1$ or $i>0$. The only solution is $i=0$ and $j=1$, giving the set of monomials $\{m_{1,1},m_{2,2},m_{2,2,2}\}$.
  \item $\{m_{2i+1,2i+1},m_{2j,2j},m_{k,k,2}\}$\\
  The determinant simplifies to
\begin{multline*}\frac{1}{48} \left(2^{2j+3}3^k+2^{2j+1}6^k\right.\\
-\left(1+(-1)^k\right)\left(2^{2i+k+1}+2^k6^{2i+1}+2^k3^{2j}+2^k3^{2i+1}
+2^k6^{2j}\right)\\
\left.-4\left(-1\right)^k\left(6^{2i+1}+3^{2i+1}+2^{2i+1}+6^{2j}+3^{2j}+2^{2j}\right)
+\left(17-(-1)^k\right)2^{2j+k}\right).
\end{multline*}
As before, we require that this determinant be $\pm 1$. So, we need to solve the two exponential Diophantine equations:
\begin{multline}\label{Dio2}
2^{2j+3}3^k+2^{2j+1}6^k\\ -\left(1+(-1)^k\right)\left(2^{2i+k+1}+2^k6^{2i+1}+2^k3^{2j}+2^k3^{2i+1}
+2^k6^{2j}\right)\\
-4\left(-1\right)^k\left(6^{2i+1}+3^{2i+1}+2^{2i+1}+6^{2j}+3^{2j}+2^{2j}\right)\\
+\left(17-(-1)^k\right)2^{2j+k}=\pm 48
\end{multline}
 Reducing (\ref{Dio2}) modulo 45 shows that there are no solutions when the right-hand side is $-48$, while reduction modulo 1197 proves that there are no solutions to (\ref{Dio2}) when the right-hand side is $48$. Hence, no spanning set arises in this situation.
\end{itemize}

We summarize the above computations on the even monomials in the following lemma.

\begin{lem}\label{lem:evenblock}
The {sets of even monomial symmetric polynomials of $L_1,\dots,L_4$ spanning $\{\underline C_\emptyset, \underline C_{1,1}, \underline C_2\}$ are:} 
\begin{itemize}
  \item $\{m_\emptyset,m_{1,1},m_2\}$
  \item $\{m_\emptyset,m_{1,1},m_{2,2,2}\}$
  \item $\{m_\emptyset,m_{1,1},m_{i,i,2}\mid i\text{ odd}\}$
  \item $\{m_{1,1},m_{2,2},m_{2}\}$
  \item $\{m_{1,1},m_{2,2},m_{2,2,2}\}$
\end{itemize}
\end{lem}

We now turn our attention to the odd monomials.

\begin{lem}\label{lem:oddmoncan} The odd {monomial symmetric polynomials of $L_1,\dots,L_4$} that have at least one odd coefficient on a class element
are:
\begin{enumerate}
\item $\{m_i\mid i\text{ odd}\}$ \label{miodd}
\item $m_{1,1,1}$ from $\{m_{i,i,i}\mid i\text{ odd}\}$
\item $m_{i,i,1}$ from $\{m_{i,i,j}\mid j\text{ odd, } j\ne i\}$ \label{miijodd}
\item $\{m_{i+j,i}\mid j\text{ odd}\}$
\item none from $\{m_{k+i+j,k+i,k}\mid j+k\text{ odd}\}$.
\end{enumerate}
\end{lem}

\begin{proof}
  The proofs are similar to the proofs given of Lemma \ref{lem:evenmoncan}. For example, for part (\ref{miodd}), observe that $3^i+4\cdot 2^i+1 \equiv 4 \pmod{8},$ which shows that $\l s_1,m_{i}\r$ is odd. Also, $3^i-2\cdot 2^i+1 \equiv 4 \pmod{8}$ when $i\ge 3$, and $3^i-2\cdot 2^i+1$ is divisible by 8 when $i=1$, which shows that $\l s_{123},m_{i}\r$ is even only when $i=1$.
\end{proof}

A consequence of Lemma \ref{lem:oddmoncan} is that the parities of the candidate odd monomials are:

\medskip
\begin{tabular}{c|ccccc}
&(1)&$(1,1,1)$&$(i,i,1)$&$(i)$&$(i+j,i)$\\
&&&$i\ge 3$& $i\ge 3$&\\
\hline
$s_{1}$&1&0&even&odd&odd\\
$s_{123}$&0&1&odd&odd&odd
\end{tabular}
\medskip

From this table we see that there are nine possibilities where the determinant of the $2 \times 2$ odd block is odd. We show below which of these possibilities actually yield a determinant equal to $\pm 1$. In certain cases, Maple was used in the computations.

\begin{itemize}
  \item $\{m_1,m_{1,1,1}\}$\\
  This is clearly a spanning set for the odd class elements.
    \item $\{m_i,m_{1,1,1}\}$ with $i\ge 3$\\
  Since we require that the determinant be $\pm 1$, we derive the following equations:
\[3^i+4\cdot 2^i+1=\pm 12.\]
It is easy to see that there is no solution since $i\ge 3$.
\item $\{m_{i+j,i}$, $m_{1,1,1}\}$\\
  Since $j$ is odd,
  the following equations are derived from requiring that the determinant be equal to $\pm 1$:
\[3^j6^i+2^j6^i+3^{i+j}+3^i+\left(4+3(-1)^i\right)2^{i+j}
+\left(4-3(-1)^i\right)2^i=\pm 12.\]
 There are no solutions since the left-hand side is easily seen to be larger than 12.
\item $\{m_1,m_{i,i,1}\}$ with $i\ge 3$\\
The determinant here is
\[\frac{1}{12}\left(6^i+2^i\cdot 3+2\cdot 3^i+3\cdot (2^i+2)+3\cdot 2^i\right),\]
which is clearly larger than 1. So, there are no solutions in this situation.
\item $\{m_1,m_{i}\}$ with $i\ge 3$\\
Requiring that the determinant be $\pm 1$ produces the equations:
\[3^{i}-2\cdot 2^{i}+1=\pm 12.\] It is easy to see that there is the single solution $i=3$. Thus, the odd block can be $\{m_1, m_3\}.$
\item $\{m_1,m_{i+j,i}\}$\\
Setting the determinant equal to $\pm 1$ gives the two equations
\[3^j6^i+2^j6^i+3^i+3^{i+j}+3(-1)^i(2^i-2^{i+j})-2(2^i+2^{i+j})=\pm 12.\] Reduction modulo 5 shows that the left-hand side is congruent to 0, 1 or 4, while the right-hand side is congruent to 2 or 3. Thus, there are no solutions here.
\item $\{m_i$, $m_{j,j,1}\}$ with $i\ge 3$\\
Setting the determinant equal to $\pm 1$ leads to two exponential Diophantine equations: one with $-48$ on the right-hand side, and one with $48$ on the right-hand side. The $-48$--equation has no solutions mod 819, while the $48$--equation has no solutions mod 5.
\item $\{m_{j+k,j}$, $m_{i,i,1}\}$\\
As above we get two exponential Diophantine equations by equating the determinant to $\pm 1$. The $-48$--equation has no solutions mod 45, while the $48$--equation has no solutions mod 85.
\item $\{m_{i,i,1},m_{j,j,1}\}$ with $i,j\ge 3$\\
In this situation, we arrive at two exponential Diophantine equations: one with $-144$ on the right-hand side, and one with $144$ on the right-hand side. Reduction modulo 91 shows that there are no solutions in either case since the left-hand side is congruent to 0 or 16 while the right-hand side is congruent to 38 or 53.
\end{itemize}

We summarize the above computations on the odd monomials in the following lemma.

\begin{lem}\label{lem:oddblock}
The {sets of odd monomial symmetric polynomials of $L_1,\dots,L_4$ that span $\{\underline C_1,\underline C_3\}$} are:
\begin{itemize}
  \item $\{m_1,m_{1,1,1}\}$
  \item $\{m_1, m_3\}.$
  \end{itemize}
\end{lem}

 Lemma \ref{lem:evenblock} and Lemma \ref{lem:oddblock} determine all bases for $Z(\Z S_4)$ which consist solely of monomial symmetric polynomials in {Jucys--Murphy} elements. We get eight specific bases and two infinite families of bases.
We state this main result in {the following Theorem}.

\begin{thm}\label{thm:S4bases}
 The complete list of bases for $Z(\Z S_4)$ which consist solely of monomial symmetric polynomials in {Jucys--Murphy} elements is:
\begin{itemize}
  \item $\left\{m_{\emptyset}, m_1, m_2, m_{1,1}, m_{1,1,1}\right\}$
  \item $\left\{m_{\emptyset}, m_1, m_{1,1}, m_{1,1,1}, m_{2,2,2}\right\}$
  \item $\left\{m_1, m_2, m_{1,1}, m_{1,1,1}, m_{2,2}\right\}$
  \item $\left\{m_1, m_{1,1}, m_{1,1,1}, m_{2,2}, m_{2,2,2}\right\}$
  \item $\left\{m_{\emptyset}, m_1, m_2, m_{1,1}, m_{3}\right\}$ \qquad {(Murphy's basis~\cite{Mur83})}
  \item $\left\{m_{\emptyset}, m_1, m_{1,1}, m_{2,2,2}, m_{3}\right\}$
  \item $\left\{m_1, m_2, m_{1,1}, m_{3}, m_{2,2}\right\}$
  \item $\left\{m_1, m_{1,1}, m_{3}, m_{2,2}, m_{2,2,2}\right\}$
  \item $\left\{m_{\emptyset}, m_1, m_{1,1}, m_{1,1,1}, m_{i,i,2}\mid i \text{ odd }\right\}$
\item $\left\{m_{\emptyset}, m_1, m_{1,1}, m_{3}, m_{i,i,2}\mid i \text{ odd }\right\}.$  \end{itemize}

\end{thm}
As mentioned in the Introduction, any integral basis for $Z(\H_n)$ specializes at $\xi=0$ to an integral basis for $Z(\Z S_n)$. So, in particular, to find all integral bases for $Z(\H_4)$ which consist solely of monomial symmetric polynomials in {Jucys--Murphy} elements, it is sufficient to check the sets of monomials in $\H_4$ corresponding to the bases given in Theorem \ref{thm:S4bases}. Checking the infinite families for {$i< 50$}, and the remaining sporadic bases gives the following list of bases for $Z(\H_4)$:
\begin{align*}
&\{m_\emptyset,m_1,m_2,m_{1,1},m_{1,1,1}\}\\
&\{m_\emptyset,m_1,m_{1,1},m_{1,1,1}, m_{2,1,1}\}\\
&\{m_\emptyset,m_1,m_{1,1},m_{1,1,1}, m_{2,2,2}\}.\\
\end{align*}

\section{{Summary}, Generalizations {and Conjectures} }\label{sec:H4+conjs}

Using GAP to search up through partitions of {10}, we have found the following 12 monomial bases for {$Z(\Z S_5)$}:

\begin{align*}
&\{m_\emptyset,m_1,m_2,m_{1,1},m_{3},m_{2,1},m_{4}\}\\
&\{m_\emptyset,m_1,m_2,m_{1,1},m_{3},m_{2,1},m_{1,1,1,1}\}\\
&\{m_\emptyset,m_1,m_2,m_{1,1},m_{3},m_{1,1,1},m_{4}\}\\
&\{m_\emptyset,m_1,m_2,m_{1,1},m_{3},m_{1,1,1},m_{1,1,1,1}\}\\
&\{m_\emptyset,m_1,m_2,m_{1,1},m_{2,1},m_{4},m_{5}\}\\
&\{m_\emptyset,m_1,m_2,m_{1,1},m_{2,1},m_{1,1,1,1},m_{5}\}\\
&\{m_\emptyset,m_1,m_2,m_{1,1},m_{1,1,1},m_{4},m_{3,1,1}\}\\
&\{m_\emptyset,m_1,m_2,m_{1,1},m_{1,1,1},m_{1,1,1,1},m_{3,1,1}\}\\
&\{m_\emptyset,m_1,m_{1,1},m_{3},m_{2,1},m_{2,1,1},m_{1,1,1,1}\}\\
&\{m_\emptyset,m_1,m_{1,1},m_{3},m_{1,1,1},m_{2,1,1},m_{1,1,1,1}\}\\
&\{m_\emptyset,m_1,m_{1,1},m_{2,1},m_{2,1,1},m_{1,1,1,1},m_{5}\}\\
&\{m_\emptyset,m_1,m_{1,1},m_{1,1,1},m_{2,1,1},m_{1,1,1,1},m_{3,1,1}\}.\\
\end{align*}
None of these sets of monomials in $\H_5$ is a basis for $Z(\H_5$).
The following table summarizes what is currently known regarding monomial bases for $Z(\Z S_n)$ and $Z(\H_n)$ when $n=3,4,5$.\\

\begin{tabular}{c|l|l}
Algebra & Number of Monomial Bases& Reference\\ \hline
$Z(\Z S_3)$ & 4 & \cite{FJ:newintbasis}\\
$Z(\H_3)$ & 1 & \cite{FJ:newintbasis}\\
$Z(\Z S_4)$ & 8 + two infinite families & Theorem~\ref{thm:S4bases} \\
$Z(\H_4)$ & 3 known & end of Section \ref{sec:bases}\\
$Z(\Z S_5)$ & 12 known & see above\\
$Z(\H_5)$ & none known & checked the 12 known\\&& bases for $Z(\Z S_5)$\\
\end{tabular}\\

\begin{conjs} We conjecture the following: 
  \begin{enumerate}
  \item There are only 12 monomial bases for $Z(\Z S_5)$.
  \item When $n\ge 5$, there are only finitely many monomial bases for $Z(\Z S_n)$.
  \item {There are only 3 monomial bases for $Z(\H_4)$.}
  \item When $n\ge 5$, there are no monomial bases for $Z(\H_n)$.
  \end{enumerate}
\end{conjs}


\begin{thebibliography}{10}

\bibitem{DJ87}
Richard Dipper and Gordon James.
\newblock Blocks and idempotents of {H}ecke algebras of general linear groups.
\newblock {\em Proc. London Math. Soc. (3)}, 54(1):57--82, 1987.

\bibitem{FH59}
H.~K. Farahat and G.~Higman.
\newblock The centres of symmetric group rings.
\newblock {\em Proc. Roy. Soc. London Ser. A}, 250:212--221, 1959.

\bibitem{Fmb}
Andrew Francis.
\newblock The minimal basis for the centre of an {I}wahori-{H}ecke algebra.
\newblock {\em J. Algebra}, 221(1):1--28, 1999.

\bibitem{FJ:newintbasis}
Andrew Francis and Lenny Jones.
\newblock A new integral basis for the centre of the {Hecke} algebra of type
  {$A$}.
\newblock arXiv:0705.1581.

\bibitem{FG:DJconj2006}
Andrew~R. Francis and John~J. Graham.
\newblock {Centres of Hecke algebras: the Dipper-James conjecture}.
\newblock {\em J. Algebra}, 306:244--267, 2006.

\bibitem{GHLMP}
M.~Geck, G.~Hiss, F.~L{\accent127 u}beck, G.~Malle, and G.~Pfeiffer.
\newblock {CHEVIE}---a system for computing and processing generic character
  tables. {Computational} methods in {Lie} theory ({Essen} 1994).
\newblock {\em Appl.~Algebra ENGRG.~Comm.~Comput.}, 7(3):175--210, 1996.

\bibitem{GR97}
Meinolf Geck and Rapha{\"e}l Rouquier.
\newblock Centers and simple modules for {I}wahori-{H}ecke algebras.
\newblock In {\em Finite reductive groups (Luminy, 1994)}, pages 251--272.
  Birkh\"auser Boston, Boston, MA, 1997.

\bibitem{Jucys1966}
A.~A.~A. Jucys.
\newblock On the young operators of symmetric groups.
\newblock {\em Litovsk. Fiz. Sb.}, 6:163--180, 1966.

\bibitem{Jucys1971}
A.~A.~A. Jucys.
\newblock Factorization of {Y}oung's projection operators for symmetric groups.
\newblock {\em Litovsk. Fiz. Sb.}, 11:1--10, 1971.

\bibitem{Jucys1974}
A.~A.~A. Jucys.
\newblock Symmetric polynomials and the center of the symmetric group ring.
\newblock {\em Rep. Mathematical Phys.}, 5(1):107--112, 1974.

\bibitem{Macdonald95}
I.~G. MacDonald.
\newblock {\em Symmetric functions and Hall polynomials}.
\newblock Oxford Mathematical Monographs. Oxford University Press, second
  edition, 1995.

\bibitem{Maple10}
Michael~B. Monagan, Keith~O. Geddes, K.~Michael Heal, George Labahn, Stefan~M.
  Vorkoetter, James McCarron, and Paul DeMarco.
\newblock {\em Maple~10 Programming Guide}.
\newblock Maplesoft, Waterloo ON, Canada, 2005.

\bibitem{Mur83}
G.~E. Murphy.
\newblock The idempotents of the symmetric group and {N}akayama's conjecture.
\newblock {\em J. Algebra}, 81:258--265, 1983.

\bibitem{Sch95}
Martin Sch{\accent127 o}nert et~al.
\newblock {\em {GAP} --- {Groups}, {Algorithms}, and {Programming}}.
\newblock Lehrstuhl D f{\accent127 u}r Mathematik, Rheinisch Westf{\accent127
  a}lische Technische Hoch\-schule, Aachen, Germany, fifth edition, 1995.

\end{thebibliography}
\end{document}